\documentclass[a4paper,11pt]{amsart}
\title[Signs of knot pol.\ eval.\ from a topol.\ perspective]{Signs of knot polynomial evaluations from a topological perspective}
\author{Luana Jost}
\address{Institut de Math\'ematiques de Toulouse, CNRS, Universit\'e de Toulouse, 118, route de Narbonne, 31062 Toulouse Cedex 9, France}
\email{\myemail{ljost@math.univ-toulouse.fr}}
\author{Lukas Lewark}
\address{ETH Z\"urich, R\"amistrasse 101, 8092 Z\"urich, Switzerland}
\email{\myemail{llewark@math.ethz.ch}}
\urladdr{\url{https://people.math.ethz.ch/~llewark/}}
\usepackage[latin1]{inputenc}
\usepackage{amsmath,amssymb,stmaryrd,thmtools,enumerate,graphicx,xcolor,enumitem,amscd,url,mathtools,microtype,tikz-cd,caption,amsthm}
\usepackage[bookmarks=true,pdfencoding=auto,hidelinks]{hyperref}
\hypersetup{pdfauthor={\authors},pdftitle={\shorttitle}}
\usepackage[nameinlink]{cleveref}
\let\cref\Cref
\crefname{subsection}{section}{sections}
\Crefname{subsection}{Section}{Sections}
\Crefname{enumi}{}{}
\crefname{equation}{}{}
\newcommand{\myemail}[1]{\href{mailto:#1}{#1}}
\newcommand{\qua}{\hskip 0.4em \ignorespaces}
\def\arxiv#1{\relax\ifhmode\unskip\qua\fi
\href{http://arxiv.org/abs/#1}%
{\tt arXiv:\penalty -100\unskip#1}}

\def\MR#1{\relax\ifhmode\unskip\qua\fi
\href{https://mathscinet.ams.org/mathscinet-getitem?mr=#1}{\tt MR#1}}
\def\ZB#1{\relax\ifhmode\unskip\qua\fi
\href{https://zbmath.org/?q=an:#1}{\tt Zbl\:#1}}
\def\xox#1{\csname xx#1\endcsname}

\renewenvironment{thebibliography}[1]{
  \begin{oldthebibliography}{#1}\small
    \setlength{\itemsep}{.5ex}
    \setlength{\parskip}{0em}
}
{
  \end{oldthebibliography}
}
\let\stdthebibliography\thebibliography
\let\stdendthebibliography\endthebibliography
\renewenvironment*{thebibliography}[1]{%
  \stdthebibliography{XXXXX.}}
  {\stdendthebibliography}
\numberwithin{equation}{section}
\newtheorem{lemma}{Lemma}[section]
\newtheorem{theorem}[lemma]{Theorem}
\newtheorem{corollary}[lemma]{Corollary}
\newtheorem{proposition}[lemma]{Proposition}
\newtheorem{conjecture}[lemma]{Conjecture}
\theoremstyle{definition}
\newtheorem{definition}[lemma]{Definition}
\newtheorem{remark}[lemma]{Remark}
\newtheorem{example}[lemma]{Example}
\newcommand{\Z}{\mathbb{Z}}
\newcommand{\Q}{\mathbb{Q}}
\newcommand{\R}{\mathbb{R}}

\newcommand{\legendre}[2]{\ensuremath{\biggl(\frac{#1}{#2}\biggr)}}
\newcommand{\legendres}[2]{\ensuremath{\bigl(\frac{#1}{#2}\bigr)}}
\DeclareMathOperator{\Arf}{Arf}
\DeclareMathOperator{\coker}{coker}
\DeclareMathOperator{\ord}{ord}
\DeclareMathOperator{\lk}{lk}

\usepackage[nice]{nicefrac}

\begin{document}
\begin{abstract}
We prove that for knots, the evaluation of the Jones polynomial at the sixth root of unity, as well as the evaluation of the $Q$\nobreakdash-polynomial at the reciprocal of the golden ratio, are uniquely determined by the oriented homeomorphism type of the double branched covering. We provide explicit formulae for these evaluations in terms of the linking pairing. The proof proceeds via so-called singular determinants, from which we also extract new lower bounds for the unknotting numbers of knots and links.
\end{abstract}
\maketitle
\section{Introduction}
The \emph{Jones polynomial} $V_L(t) \in \Z[t^{\pm 1/2}]$ of a link $L \subset S^3$
may be defined by
the value $V(t) = 1$ for the unknot~ and
 the skein relation~\cite{zbMATH03899758}
\begin{equation}\label{eq:skeinjones}
t^{-1}V_{L_+}(t) - tV_{L_-}(t)
=
(t^{1/2} - t^{-1/2})V_{L_0}(t),
\end{equation}
where $L_+, L_-, L_0$ are oriented links that locally differ as follows:
\begin{figure}[h]
\centering
\includegraphics[scale=.25]{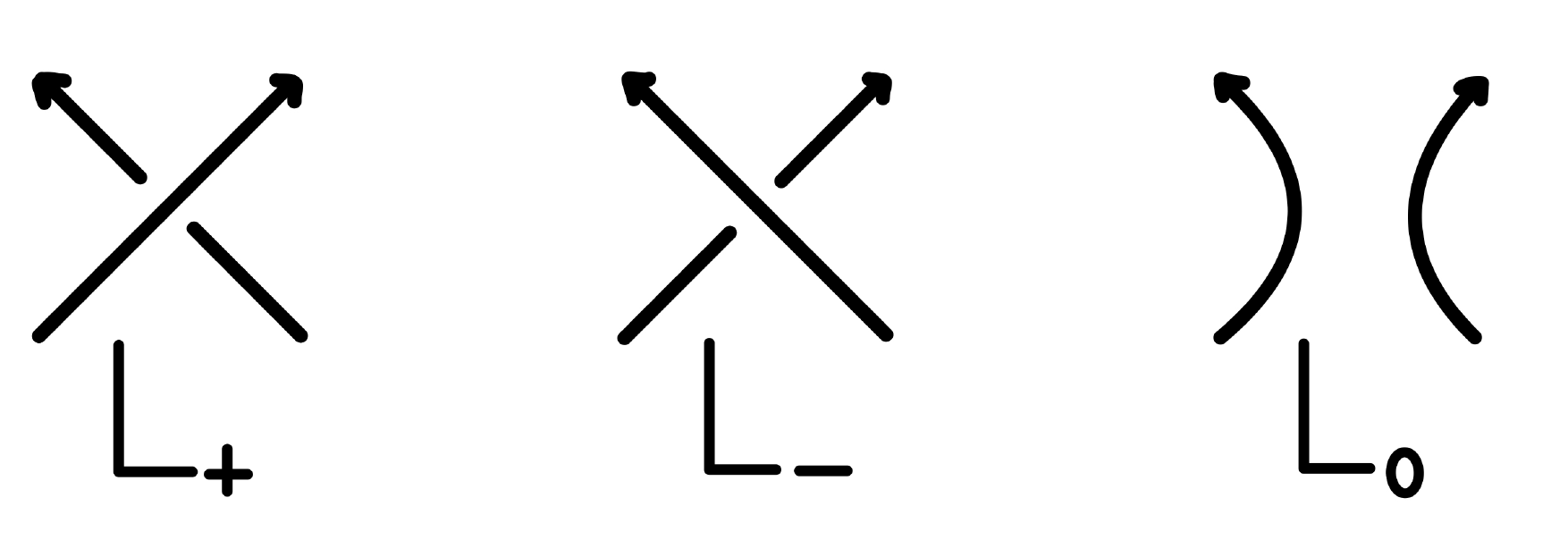}
\end{figure}

It is a fundamental open problem to find a 3-dimensional topological interpretation of the Jones polynomial~\cite{zbMATH04092365,zbMATH04092352}.
Such an interpretation does, however, exist for the evaluation of the Jones polynomial at the
primitive first, second, third, fourth and sixth roots of unity.
For simplicity, let us restrict ourselves to knots (the case of links will be discussed in \cref{sec:links}); note that $V_K(t) \in \Z[t^{\pm 1}]$ for a knot~$K$. The following hold~\cite{zbMATH03899758,zbMATH03982072}:
\begin{align}
\notag V_K(1)  & = V_K(e^{2\pi i/3}) = 1, \\
\notag V_K(-1) & = (-1)^{(\det K - 1)/2} \cdot \det(K), \\
\notag V_K(i) & = (-1)^{\Arf(K)}, \\
V_K(e^{2\pi i/6}) & = \varepsilon_K \cdot (i\sqrt{3})^{\dim H_1(D_K; \mathbb{F}_3)}.\label{eq:atomega}
\end{align}
Here, the \emph{determinant} $\det(K)$ is an odd positive integer, defined as the order of the first homology group $H_1(D_K; \mathbb{Z})$ of the double branched covering $D_K$ of $S^3$ along~$K$; the invariant $\Arf(K)$ equals $1$ if $\det K \equiv \pm 1\pmod{8}$, and equals $-1$ if $\det K \equiv \pm 3\pmod{8}$; and $\varepsilon_K \in \{\pm 1\}$.

The reader may have noted that the above formulae for $V_K(-1)$, $V_K(i)$, and $V_K(e^{2\pi i/6})$ depend only on $D_K$, with the potential exception of the sign~$\varepsilon_K$.
Our first main result closes this gap by providing a characterization of~$\varepsilon_K$ in terms of $H_1(D_K; \mathbb{Z})$ and its linking pairing.
Since $H_1(D_K; \mathbb{F}_3) \cong H_1(D_K; \mathbb{Z}) \otimes \mathbb{F}_3$,
it follows that $V_K(e^{2\pi i/6})$ is fully determined by $H_1(D_K; \mathbb{Z})$ and its linking pairing.

Here, a linking pairing is a bilinear non-degenerate form $H_1(M; \mathbb{Z}) \times H_1(M; \mathbb{Z}) \to \mathbb{Q}/\mathbb{Z}$ that is defined for all rational homology 3-spheres~$M$.
If the underlying group $H_1(M; \mathbb{Z})$ is of odd order (as is the case for ${M = D_K}$), then the isometry type of such pairings is determined by
isometry invariants ${r_{p,k}\in\Z/2}$, where $p$ ranges over odd primes and $k$ over positive integers~\cite{Wall}. We will discuss this in detail in \cref{sec:lkform}. For now, suffice it to say that the $r_{p,k}(K)$ are invariants of the knot~$K$ determined by~$D_K$.
Our first main result is:
\begin{restatable}{theorem}{MainTheorem}
\label{Main-Theorem}
Let a knot $K$ be given. Denote by $\alpha \geq 0$ and $q \geq 1$ the unique integers such that $3$ does not divide $q$ and $\det(K) = 3^\alpha q$.
Then
\[
V_K(e^{2\pi i/6}) = \nu(q) \cdot (-1)^{\alpha + \dim H_1(D_K; \mathbb{F}_3) + \sum_{k=1}^\infty \, r_{3,k}(K)}
\cdot (i\sqrt{3})^{\dim H_1(D_K; \mathbb{F}_3)},
\]
where for all $\eta \in \Z$ not divisible by $2$ or $3$,
\begin{align*}
    \nu(\eta) = \left\{\begin{array}{ll} 1 &\quad \eta \equiv \pm 1 \pmod{12}, \\
     -1 &\quad \eta \equiv \pm 5 \pmod{12}. \end{array}\right.
\end{align*}
\end{restatable}
Our proof of the above theorem is based on Lipson's algebraic characterization of the sign~$\varepsilon_K$ in \eqref{eq:atomega}~\cite{Lipson}. Namely, Lipson defined for all odd primes $p$ and links $L$ a \emph{singular determinant}~$\delta_p(L)\in \{\pm 1\}$ in terms of a symmetrized Seifert matrix of~$L$, and showed that $\delta_3(K) = \varepsilon_K$ for all knots~$K$ (links will be discussed in \cref{sec:links}).
Another formula for $\varepsilon_K$ in terms of the symmetrized Seifert matrix is given in~\cite{zbMATH04111602}.
The singular determinant~$\delta_5$ was independently defined by Rong~\cite{zbMATH04184437} to analyze the $Q$-polynomial.

Our second main result features an obstruction for the unknotting number of knots and links
involving the singular determinants~$\delta_p(L)$
and one of the oldest lower bounds for the unknotting number of a knot $K$ 
given by $u(K) \geq \dim H_1(D_K; \mathbb{F}_p)$~\cite{zbMATH03026573}.
This bound can easily be generalized to the bound
\begin{equation}\label{eq:wendt}
u(L) \geq \dim H_1(D_L; \mathbb{F}_p) - c(L) + 1,
\end{equation}
that holds for all links~$L$%
\footnote{There is a bijection between $\mathbb{F}_p \oplus H_1(D_L; \mathbb{F}_p)$ and the set of Fox $p$-colorings of~$L$, and so \cref{eq:wendt} may also be stated as follows:
$u(L) \geq \log_p(\text{col}_p(L)) - c(L)$, where $\text{col}_p(L)$ is the number of Fox $p$-colorings of~$L$.}.
Here, the \emph{unknotting number}~$u(L)$ of a link~$L$ is the minimal number of crossing changes necessary to unknot~$L$, i.e.~to turn it into an unlink.
Now, our following theorem makes it possible to improve the lower bound \eqref{eq:wendt} by~$1$.
\begin{restatable}{theorem}{signedu}\label{thm:signedu}
Let $p$ be an odd prime 
and $L$ a link with $c(L)$ components, such that the unknotting number $u(L)$ of $L$ satisfies
\[
u(L) = \dim H_1(D_L; \mathbb{F}_p) - c(L) + 1.
\]
Let $u(L) = u_+ + u_-$ such that $L$ can be unknotted
by changing $u_{+}$ positive crossings and $u_{-}$ negative crossings.
Then, depending on the remainder of $p$ modulo~8, one of the following equalities holds.
\begin{align*}
& p \equiv 1 \,\Longrightarrow\, \delta_p(L) = 1,
& p \equiv 3 \,\Longrightarrow\, \delta_p(L) = (-1)^{u_{-}},\\
& p \equiv 5 \,\Longrightarrow\, \delta_p(L) = (-1)^{u(L)},
& p \equiv 7 \,\Longrightarrow\, \delta_p(L) = (-1)^{u_{+}}.
\end{align*}
\end{restatable}
As special cases, this theorem recovers two previously known results:
for $L$ a knot with $u(L) = 1$, \cref{thm:signedu} is equivalent to a classical obstruction by Lickorish~\cite{zbMATH03982071,zbMATH04048673}, see \cref{rmk:lickorish};
and for $p = 3$ and $L$ a knot, \cref{thm:signedu} was established by Traczyk~\cite{zbMATH01341751}, phrased using $\varepsilon_K$ instead of $\delta_3$ (see also \cref{cor:traczyk});

For primes $p\equiv 1\pmod{4}$, \cref{thm:signedu} can be used on its own to show 
$u(L) \geq \dim H_1(D_L; \mathbb{F}_p) - c(L) + 2$ for a given link~$L$.
For primes $p\equiv 3\pmod{4}$, one may combine \cref{thm:signedu} with 
a complementary restriction on the signs of crossing changes coming from another toolkit. This
strategy was successfully implemented using Heegaard--Floer homology in \cite{zbMATH05242998,zbMATH06492012} for $p = 3$ to show that the pretzel knot $P(3,3,3)$ has $u\geq 3$
(using Traczyk's obstruction, which appears as special case of \cref{thm:signedu}).

The paper is structured as follows. 
In \cref{sec:singdet}, we analyze singular determinants of links and prove \cref{thm:signedu}.
\cref{sec:lkform} provides a characterization of singular determinants in terms of linking forms, see \cref{prop:singdetwall}.
In \cref{sec:eval}, this characterization is then used to establish \cref{Main-Theorem} and an analogous result for the $Q$-polynomial (\cref{thm:q}).
As a consequence, we find a counterexample to a conjecture by Stoimenow, see
\cref{ex:stoimenow}.
Finally, \cref{sec:links} discusses which of our results can be extended from knots to links.

\subsection*{Acknowledgments}
The authors warmly thank David Cimasoni for comments on a first version of the article.

\section{Singular link determinants}
\label{sec:singdet}
In this section, we analyze
link invariants called singular determinants, which were introduced by
Lipson \cite{Lipson}, and establish how they behave under crossing changes.
The definition of singular determinant involves a basic number theoretic construction, namely the \emph{Legendre symbol}
\[
\legendre{a}{p} \in \{-1,0,1\},
\]
which we briefly review now.
The Legendre symbol is defined for all $a\in\mathbb{Z}$ and all odd primes $p$, taking the value $0$ if $p$ divides $a$, $1$ if $p$ does not divide $a$ and $a$ is a quadratic residue modulo $p$
(i.e.~$a \equiv b^2 \pmod{p}$ for some $b\in\mathbb{Z}$), and $-1$ if $p$ does not divide $a$ and $a$ is a quadratic non-residue modulo $p$. The Legendre symbol is multiplicative, i.e.
\[
\legendre{ab}{p} = \legendre{a}{p} \legendre{b}{p}.
\]
In this article, we will not need the so-called quadratic reciprocity of the Legendre symbol.
For fixed values of $a$, there are closed formulae for the Legendre symbol, such as the following:
\begin{align}
\label{eq:legendre-1}
\legendre{-1}{p} & = \left\{\begin{array}{rl}
1  & \text{if } p \equiv 1 \pmod{4}, \\ 
-1 & \text{if } p \equiv 3 \pmod{4},
\end{array}\right.
\\[1ex]
\label{eq:legendre2}
\legendre{2}{p} & = \left\{\begin{array}{rl}
1  &\text{if }  p \equiv \pm 1 \pmod{8}, \\ 
-1 &\text{if }  p \equiv \pm 3 \pmod{8}.
\end{array}\right.
\end{align}
The Legendre symbol encodes quadratic residues not just modulo primes, but also modulo arbitrary integers~$q \geq 1$. Indeed, if $a$ is coprime with~$q$, then $a$ is a quadratic residue modulo $q$ if and only if $\legendres{a}{p} = 1$ for all prime divisors $p$ of~$q$. In particular, $a$ is a quadratic residue modulo a prime power $p^k$ if and only if~$\legendres{a}{p} = 1$.

Let us now define the singular determinant; first for arbitrary matrices, and then for links, via their Seifert matrices.
\begin{definition}
Let $M \in \mathbb{Z}^{n\times n}$ be symmetric with even diagonal entries.
Let $p$ be an odd prime.
Pick $T \in \mathbb{Z}^{n\times n}$ with $\det T = \pm 1$ such
that $TMT^{\top}$ has the following shape:
it has an upper left square submatrix $N$ with determinant not divisible by~$p$, and all its entries outside of $N$ are divisible by~$p$
(here, $N$ may be the $0\times 0$ matrix, which has determinant $1$ by convention).
In other words, modulo $p$, $TMT^{\top}$ equals the block sum of $N$ and a zero matrix.
Denote by $\mu = \mu(M)$ the dimension of the kernel of $M$ over $\mathbb{F}_2$ plus one,
and by $d_p = d_p(M)$ the dimension of the kernel of $M$ over $\mathbb{F}_p$.
Note that $N$ has size $n - d_p$.
Then the \emph{singular determinant} $\delta_p(M)$ of $M$ is defined as
\[
\delta_p(M) = \legendre{(-1)^{d_p + (n + \mu - 1)/2}\cdot\det N}{p}
\in \{-1,+1\}.
\]
\end{definition}
It follows from the assumption that the diagonal entries of $M$ are even that $(n - \mu + 1)/2$ is an integer.
For $\delta_p(M)$ to be well-defined, one needs to moreover check that a matrix~$T$ as above exists, and that $\delta_p(M)$ does not depend on the choice of $T$. For $p = 3$, this has been done by Lipson, who also notes that his proof easily adapts to all odd primes~\cite[Note~(2)]{Lipson}%
\footnote{The formula given in \cite[Note~(2)]{Lipson} contains a typo: $i$ should be $-1$.}.

A \emph{Seifert surface} of an oriented link $L\subset S^3$ is
an oriented connected compact surface $\Sigma \subset S^3$ with $\partial \Sigma = L$. For a simple closed curve $\beta \subset \Sigma$,
denote by $\beta^+ \subset S^3\setminus\Sigma$ the push-off of $\beta$ in the positive normal direction. Then the \emph{symmetrized Seifert form} of $\Sigma$ is the symmetric bilinear form $H_1(\Sigma;\mathbb{Z}) \times H_1(\Sigma;\mathbb{Z}) \to \mathbb{Z}$ given for all simple closed curves $\beta,\gamma\subset\Sigma$ by
\[
([\beta], [\gamma]) \mapsto \lk(\beta, \gamma^+) + \lk(\beta^+, \gamma).
\]
Choosing a basis of $H_1(\Sigma;\mathbb{Z})$ yields
a symmetric matrix with even diagonal entries for this bilinear form, called a \emph{symmetrized Seifert matrix} of $\Sigma$ (and of~$L$). If $M$ and $M'$ are symmetrized Seifert matrices of~$L$, then they can be obtained from each other by a finite sequence of basis changes, i.e.\ $M \rightsquigarrow TMT^{\top}$ for $\det T = \pm 1$, and transformations of the form
\begin{equation}\label{eq:Seq}
M \leftrightsquigarrow \left(\begin{array}{c|cc} M & 0 & 0 \\\hline 0 & 0 &  1 \\ 0 & 1 & 0 \end{array}\right).
\end{equation}
This follows from the well-known theorem that (non-symmetrized) Seifert matrices of the same link are S-equivalent.
Since transformations of the form~\eqref{eq:Seq} do not change~$\delta_p(M)$~\cite{Lipson},
we obtain the following well-defined link invariant~$\delta_p(L)$.
\begin{definition}\label{def:singdetL}
The \emph{singular determinant} $\delta_p(L)$ of a link $L$ is defined as~$\delta_p(M)$, for $M$ any symmetrized Seifert matrix of~$L$.
\end{definition}

When $M$ is a symmetrized Seifert matrix of~$L$, then
$\mu(M)$ equals the number of components of $L$, i.e.\ $\mu(M) = c(L)$,
and $d_p(M) = d_p(L) = \dim(H_1(D_L; \mathbb{F}_p))$.

\begin{remark}
The exponent of $(-1)$ in the definition of $\delta_p(M)$
is the sum of three terms:
\begin{itemize}
\item $n/2$, which assures invariance of $\delta_p(M)$ under \eqref{eq:Seq};
\item $(-\mu+1)/2$, which leads to a value $\delta_p(L) = 1$ for an unlink $L$;
\item $d_p + \mu - 1$, which results in $\delta_3(K) = \varepsilon_K$ for knots~$K$.
\end{itemize}
\end{remark}
\begin{remark}\label{rmk:13}
From the multiplicativity of the Legendre symbol and \eqref{eq:legendre-1},
it follows that
\begin{alignat*}{3}
\delta_p(M) & = \displaystyle\legendre{\det M}{p} & \quad \text{if } p \equiv 1\pmod{4}, \\
\delta_p(M) & = (-1)^{d_p + (n + \mu - 1)/2}\cdot\legendre{\det M}{p}
&  \quad\text{if } p \equiv 3\pmod{4}.
\end{alignat*}
\end{remark}
\begin{remark}
One can also define the singular determinant $\delta_p(L)$ of a link $L$ in terms of a Gordon--Litherland matrix~$R$.
Such a matrix is associated with a (not necessarily orientable or connected) spanning surface $S$ of~$L$~\cite{zbMATH03607187}.\footnote{%
If $S$ is a checkerboard surface coming from a diagram of $L$, then $R$ is called a \emph{Goeritz matrix}~\cite{zbMATH02546997}.}
If $S$ is orientable and connected, i.e.~a Seifert surface of $L$, then $R$ is a symmetrized Seifert matrix of $S$.
If $R$ and $R'$ are two Gordon--Litherland matrices of $L$, then they can be obtained from each other by a finite sequence of
basis changes and transformations of the form~\cite[Theorem~11]{zbMATH03607187} 
\[
R \leftrightsquigarrow \left(\begin{array}{c|c} R & 0 \\\hline 0 &  1 \end{array}\right),\qquad
R \leftrightsquigarrow \left(\begin{array}{c|c} R & 0 \\\hline 0 & -1 \end{array}\right),\qquad
R \leftrightsquigarrow \left(\begin{array}{c|c} R & 0 \\\hline 0 &  0 \end{array}\right).
\]
Trying to use the same formula as above to define $\delta_p(R)$, one runs into the problem that $n + \mu - 1$  need not necessarily be even (since $R$ may have odd diagonal entries). Moreover, the first two transformations do not leave the formula invariant.
This may be fixed by introducing another correction term, as we describe now.
Let $R \in \mathbb{Z}^{n\times n}$ be symmetric. A vector $v\in \mathbb{Z}^n$ is called \emph{characteristic} for $R$ if $w^{\top} R w \equiv v^{\top} R w \pmod{2}$ for all $w\in\mathbb{Z}^n$. Equivalently, $v_i \equiv R_{ii} \pmod{2}$ for all~$i\in\{1,\ldots,n\}$. Let us define the \emph{oddity} $o(R) \in \mathbb{Z}/8$ as $v^{\top} R v \pmod{8}$, for any characteristic vector $v$.
An elementary calculation shows that $o(R)$ is indeed well-defined modulo 8 (see e.g.~\cite{zbMATH01224949}), although for our application, $o(R) \pmod{4}$ suffices.
Now, one easily checks that
\[
\delta_p(R) = \legendre{(-1)^{d_p + (n + \mu - 1 - o(R))/2}\cdot\det N}{p}
\in \{-1,+1\}
\]
is invariant under the above three transformations. Furthermore, it equals our previous definition of $\delta_p(R)$
if $R$ has even diagonal entries (and thus~${o(R)\equiv 0}$).
Thus this formula may be used to define the singular determinants of links in terms of spanning surfaces. One observes that for disconnected surfaces~$S$, $\mu(R)$ equals $c(L)$ minus the number of components of $S$ plus one.
Note that \cref{rmk:13} again applies, and so for $p\equiv 1 \pmod{4}$,
\[
\delta_p(R) = \legendre{\det N}{p}.
\]
Indeed, Rong~\cite{zbMATH04184437} used Gordon--Litherland matrices to define~$\delta_5$.
\end{remark}

For the proof of \cref{thm:signedu}, we will need the following Lemma.
\begin{lemma}\label{lemma:goodmatrices}
Let the link $L_-$ be obtained from the link $L_+$ by changing a positive crossing.  Let $p$ be an odd prime.
Then there are symmetrized Seifert matrices $M_+$, $M_-$ for $L_+$, $L_-$, respectively, satisfying the following. The matrices are of the same size and identical except for the last diagonal entry, which is greater by two in $M_-$. Moreover, modulo $p$ the matrices decompose as block sums:
either it holds that
\begin{equation}\label{lemma:case1}
M_{\pm} \equiv P \oplus (a \mp 1),
\end{equation}
or it holds that
\begin{equation}\label{lemma:case2}
M_{\pm} \equiv P \oplus \begin{pmatrix} 0 & 1 \\ 1 & a \mp 1 \end{pmatrix},
\end{equation}
for some integer $a$ and a matrix $P$.
\end{lemma}

\begin{figure}[b]
\centering
\includegraphics[scale=.45]{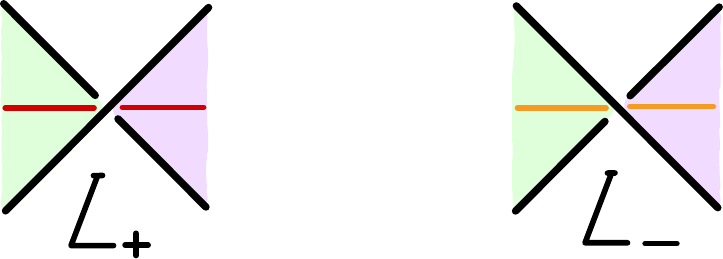}
\caption{Seifert surfaces for $L_+$ and $L_-$.}
\label{fig:seifert}
\end{figure}

\begin{proof}
By a suitable choice of Seifert surfaces and bases of their first homology groups (see \cref{fig:seifert}), one obtains symmetrized Seifert matrices $M_+$, $M_-$ for $L_+$, $L_-$ respectively of the form 
\[
M_{\pm} = \begin{pmatrix}
P_1 & v_1^{\top} \\
v_1 & a_1\mp 1
\end{pmatrix}.
\]
To see this, assume that a curve $\beta$ runs along the band as shown in \cref{fig:seifert}. Then, in the case of a positively twisted band (i.e.~twisted such that the corresponding crossing is positive), the twist decreases $lk(\beta, \beta^+)$ by $\frac{1}{2}$, while for a negative twisted band, it increases $\lk(\beta, \beta^+)$ by $\frac{1}{2}$. Therefore, the linking number $\lk(\beta, \beta^+)$ 
for $L_+$ is one smaller than that for~$L_-$. The symmetrized Seifert matrices then differ by two in the last diagonal entry.

To achieve one of the conditions modulo $p$, we make a series of basis changes.
After a first change, the matrix modulo $p$ is of the form 
\[
\begin{pmatrix}
P_2 & 0 & v_2^{\top} \\
0   & 0 & v_3^{\top} \\
v_2 & v_3 & a_1 \mp 1
\end{pmatrix}
\]
with $\det P_2 \not\equiv 0\pmod{p}$. A second change leads to 
\[
\begin{pmatrix}
P_2 & 0 & 0 \\
0   & 0 & v_3^{\top} \\
0 & v_3 & a_2 \mp 1
\end{pmatrix}.
\]
Note that $v_3$ is a potentially empty vector.
If $v_3$ is the zero vector modulo $p$, we find ourselves in case \cref{lemma:case1}, and otherwise (after a further basis change) in case \cref{lemma:case2}.
\end{proof}
We are now in the position to prove \cref{thm:signedu}, which we recall for the reader's convenience.
\signedu*
\begin{proof}
We proceed by induction on $u(L)$. If $u(L) = 0$,
then $L$ is the unlink. A symmetrized $n\times n$ Seifert matrix for $L$ is then given by the zero matrix of size $n = d_p(L) = c(L) - 1$.
One computes $\delta_p(L) = 1$. Since $u_{+} = u_{-} = 0$, this implies the desired equalities for all the four possible remainders of $p$ modulo~8.

If $u(L) \geq 1$, then let $L'$ be the link obtained by performing the first crossing change in a sequence that transforms $L$ into the unlink. Thus $u(L') = u(L) - 1$. Since a crossing change alters $d_p$ by at most $1$, it follows that $d_p(L') = d_p(L) - 1$. Thus $L'$ satisfies the induction hypothesis.

Let us first consider the case that $L' = L_-$ comes from $L = L_+$ by changing a positive crossing to a negative one.
Let us pick matrices $M_+$ and $M_-$ as in \cref{lemma:goodmatrices}.
Note that $d_p(L') = d_p(L) - 1$ implies that, over $\mathbb{F}_p$, the rank of $M_-$ is one greater than the rank of $M_+$.
In case \cref{lemma:case2}, the ranks of $M_-$ and $M_+$ over $\mathbb{F}_p$ would be equal (independently of the value of~$a$). So we must be in case \cref{lemma:case1} with $a \equiv 1$ modulo~$p$.
By the multiplicativity of the Legendre symbol, and because $d_p(L') = d_p(L) - 1$,
\[
\delta_p(L) = \delta_p(L') \cdot \legendre{-2}{p}.
\]
The analogous argument in the case that $L'$ comes from $L$ by changing a negative crossing yields
\[
\delta_p(L) = \delta_p(L') \cdot \legendre{2}{p}.
\]
Now, the result follows for each of the four possible remainders of $p$ modulo~8
by applying the induction hypothesis and~\eqref{eq:legendre-1},~\eqref{eq:legendre2}.
\end{proof}
\begin{example}
The pretzel knot $K=P(7,-7,7)$ has a Seifert matrix
$
\bigl(\begin{smallmatrix}
0	&	7\\
7	&	0
\end{smallmatrix}\bigr)
$,
and hence
$d_7 = 2$ and~$\delta_7 = 1$. It follows that if $P(7,-7,7)$ can be unknotted with two crossing changes, then they must be of the same sign.
We are not aware of any other method to establish this fact.
The best known bounds for the unknotting number $u(K)$ of $P(7,-7,7)$ are $2 \leq u(K) \leq 7$.
One possible strategy to prove $3 \leq u(K)$ would be to use a different obstruction
to show that $P(7,-7,7)$ cannot be unknotted with two crossing changes of opposite sign,
compare~\cite{zbMATH05242998,zbMATH06492012}.
\end{example}
\begin{example}
Let $K$ be a knot with Seifert matrix
\[
M = 
\begin{pmatrix}
0	&	17	&	0	&	0 \\
17	&	0	&	0	&	0 \\
0	&	0	&	6	&	3 \\
0	&	0	&	3	&	10
\end{pmatrix},
\]
such as $K = P(17,-17,17) \# P(3,3,7)$.
Then $d_{17} = 3$ and $\delta_{17} = -1$. It follows that $u(K) \geq 4$.
As in the previous example, we do not know of a different way to prove this.
\end{example}
\begin{remark}
By construction, \cref{thm:signedu} remains valid when one replaces $u(L)$ with the minimal Gordian distance between $L$ and a link $L'$ with trivial $H_1(D_{L'}; \mathbb{F}_p$) (and $u_{\pm}$ with the number of positive and negative crossing changes necessary to reach $L'$ from $L$).
In particular, for knots, the obstructions in \cref{thm:signedu} apply to the so-called algebraic unknotting number introduced by Fogel~\cite{fogel}, which is the Gordian distance to a knot with Alexander polynomial~$1$; see also~\cite{zbMATH06425399}.
\end{remark}

\section{Linking forms}\label{sec:lkform}
In this section, we relate the singular determinant of a matrix with the linking form the matrix presents, see \cref{prop:singdetwall} below.
Let us start by reviewing linking forms.
The first homology group $H_1(X; \mathbb{Z})$ of a rational homology 3-sphere $X$ can be equipped with a so-called \emph{linking form} (see e.g.~\cite{zbMATH06958506})
\[
\lambda_X\colon H_1(X; \mathbb{Z})\times H_1(X; \mathbb{Z}) \to \mathbb{Q}/\mathbb{Z}.
\]
This is a bilinear symmetric form. It is non-singular, i.e.~for all non-zero $x\in H_1(X; \mathbb{Z})$ there exists $y\in H_1(X; \mathbb{Z})$ such that $\lambda_X(x,y) \neq 0$.
The linking form $\lambda_X(x,y)$ may be defined as $1/a$ times the intersection number of $x$ with a $2$-chain $z$, such that the boundary of $z$ is $ay$ for some non-zero integer~$a$.
Abstractly, two bilinear symmetric forms $\lambda_G, \lambda_H$ on abelian groups $G$, $H$ with target $\mathbb{Q}/\mathbb{Z}$ are called \emph{isometric} if there is an isomorphism $\varphi\colon G\to H$ such that $\lambda_G(x,y) = \lambda_H(\varphi(x), \varphi(y))$ for all $x,y\in G$.
The isometry type of the linking form of a manifold $X$ as above is invariant under orientation preserving homeomorphisms.
Thus the classification of isometry types of bilinear forms as above
is important for the understanding of 3-manifolds.
There are several formulations of such a classification~\cite{zbMATH03011631,zbMATH03060534,Wall}; in this paper, we will follow Wall's~\cite{Wall}.
Our primary interest is the case that $X$ is the double branched covering of $S^3$ along a knot; in this case, $H_1(X; \mathbb{Z})$ is a finite abelian group of odd order. Thus, we will ignore the case of finite abelian groups of even order, for which the classification is somewhat more intricate.
\begin{theorem}[\cite{Wall}]
    Denote by $\mathcal{M}$ the commutative monoid of non-singular symmetric bilinear forms on finite abelian groups of odd order, considered up to isometry. The monoid $\mathcal{M}$ is generated by $A_{p^k}$ and $B_{p^k}$, where $p$ ranges over the odd primes, $k$ ranges over the positive integers, and
    \begin{align*}
        A_{p^k} &\colon \Z/p^k \times \Z/p^k \rightarrow \nicefrac{\Q}{\Z}, (1,1) \mapsto a/p^k,\\
         B_{p^k} &\colon \Z/p^k \times \Z/p^k \rightarrow \nicefrac{\Q}{\Z}, (1,1) \mapsto b/p^k,
    \end{align*}
    where $a$ and $b$ are a quadratic residue and a quadratic non-residue mod $p$, respectively.
    These generators satisfy the relations $A_{p^k} \oplus A_{p^k} \cong B_{p^k} \oplus B_{p^k}$, and no further relations.
\end{theorem}

The theorem immediately implies that the homomorphism $r_{p,k} \colon \mathcal{M} \rightarrow \Z/2$, defined as the number of $A_{p^k}\pmod 2$, is well-defined. In fact, the $r_{p,k}$ taken together yield an isomorphism between $\mathcal{M}$ and $(\mathbb{Z}/2)^{\oplus \infty}$.

Let us now discuss presentation matrices (see e.g.~\cite{zbMATH06958506}). %
Every symmetric matrix $M \in \mathbb{Z}^{n\times n}$ with $\det(M)\neq 0$
gives rise to a non-singular symmetric bilinear form,
\begin{align*}
\lambda_M\colon \coker M \times \coker M & \to \mathbb{Q}/\mathbb{Z}, \\
\lambda_M([x],[y]) & = x^{\top} M^{-1} y.
\end{align*}
We call $\lambda_M$ the \emph{linking form}\footnote{Note that linking forms are defined as bilinear mappings to $\mathbb{Q}/\mathbb{Z}$, and not to $\mathbb{Z}$, as otherwise these would all be constant zero mappings. Indeed, for $x,y \in \coker M$ exists a positive integer $a$ s.t. $ax = 0$ (elements of $\coker M$ are torsion elements). By bilinearity, $0 = \lambda_M(ax,y) = a \lambda_M(x,y)$. Since $a \neq 0, \lambda_M \equiv 0$.} of~$M$,
and say that $M$ is a \emph{presentation matrix} for~$\lambda_M$.
If $T \in \mathbb{Z}^{n\times n}$ with $\det(T)=\pm 1$,
then the matrix $TMT^{\top}$ (which is obtained from $M$ by a basis change) presents a linking form $\lambda_{TMT^{\top}}$ that is isometric to~$\lambda_M$.
Presentation matrices occur rather naturally:
if $Y$ is a compact oriented connected 4-manifold whose boundary $\partial Y = X$ is a rational homology 3-sphere, then a matrix for the intersection pairing of $Y$ presents the linking form on~$X$. For knots, a symmetrized Seifert matrix presents the linking form of the double branched covering.

For the proof of \cref{prop:singdetwall}, we will need the following three lemmas.

\begin{lemma} [\textbf{General Jacobi Identity}, {\cite[Lemma~A.1~(e)]{zbMATH06234489}}]
    \label{Lemma-Jacobi}
 Let $M \in \R^{n \times n}$ be a symmetric matrix, and $I, J\subset \{1,\ldots,n\}$ subsets of indices such that $|I| = |J|$. Denote by $M[I;J]$ the matrix restricted to the rows in $I$ and columns in $J$, in the original order. Then, we have
\[
        \det(M[I;J]) = (-1)^{\sum_{i \in I}i + \sum_{j \in J} j} \det(M) \det(M^{-1}[I^{c}; J^{c}]).
\]
\end{lemma}

In what follows, for $p$ a prime we write $\ord_p\colon \mathbb{Q} \to \mathbb{Z}\cup\{\infty\}$ for the $p$-adic valuation, i.e.~for $z\in\mathbb{Q}, $$\ord_p(z)$ is $\infty$ for $z = 0$ and chosen so that $p^{-\ord_p(z)} z$ has numerator and denominator coprime with $p$ for $z\neq 0$.
By convention, $\infty > a$ for all $a\in\mathbb{Z}$.
\begin{lemma}\label{lem:rationalchange}
For all symmetric matrices $N \in \mathbb{Q}^{m\times m}$ and integers $\rho$, there exists
$S \in \mathbb{Z}^{m\times m}$ with $\det S = \pm 1$ such that 
$N' = SNS^{\top}$ satisfies $\ord_p N'_{ii} \leq \ord_p N'_{jj}$ for all $i\geq j$,
and $\ord_p N'_{ii} < \ord_p N'_{ij}$ for all $i\neq j$.
Moreover, $\rho \leq \ord_p N'_{ij}$ for all $i\neq j$.
\end{lemma}
\begin{proof}
We will construct $S$ as a product of transformation matrices
that correspond to elementary basis changes (i.e.~permuting basis vectors, and adding integer multiples of one basis vector to another).
For better readability, we just describe the basis changes, and do not explicitly write the corresponding transformation matrices.
Let us proceed by induction over~$m$.  Denote by $\omega$ the minimal $\ord_p$ of all entries in $N$.

As first case, assume there is a diagonal entry of $N$ with $\ord_p = \omega$. After a basis change (permuting basis vectors), this entry is at $(m,m)$.
After a further basis change (adding integer multiples of the $m$-th basis vector to the others), all entries $(m,j)$ with $j \neq m$ have $\ord_p > \omega$
and $\ord_p \geq \rho$.
The submatrix of $N$ obtained by deleting the last row and column is symmetric.
By induction, there exists a basis change of the first $m-1$ basis vectors, which we compose with the previous basis changes to obtain~$S$. 
This concludes the first case.

As second case, assume $\ord_p > \omega$ for all diagonal entries of $N$, and $\ord_p = \omega$ for an off-diagonal entry at $(i,j)$. After a basis change (adding the $i$-th basis vector to the $j$-th), we have $\ord_p = \omega$ for the $(j,j)$-entry, and find ourselves back in the first case. Here, we use that $p \neq 2$.
This concludes the proof.
\end{proof}
\begin{lemma}\label{Lemma-matrixtransformation2}
Let $M$ be a symmetric $n\times n$ integer matrix with $\det(M)\neq 0$ and $p$ be an odd prime number.
Denote by $0 \leq k_1 \leq \dots \leq k_n$ the unique integers such that
$\coker M \cong \Z/p^{k_1} \oplus \dots \oplus \Z/p^{k_n} \oplus \Z/q$ with $q$ not divisible by~$p$.
Then there exists an $n\times n$ integer matrix $T$ with $\det T = \pm 1$ such that the $(i,j)$-entry of the rational matrix $(TMT^{\top})^{-1}$ has
$\ord_p \geq 0$ for $i\neq j$ and $\ord_p = -k_i$ for $i = j$.
Moreover, the $(i,j)$-entry of the integer matrix $TMT^{\top}$ has
$\ord_p \geq k_i + k_j$ for $i\neq j$ and
$\ord_p \geq k_i$ for $i = j$.
\end{lemma}
Before the proof,
let us make an explicit example.
\begin{example}
Let $M$ be a symmetrized Seifert matrix of the knot 12n553.
\[
\small
M = \begin{pmatrix} 
-2 &  0 & -1 &  0\\
 0 & -6 &  9 &  3\\
-1 &  9 & -8 & -3\\
 0 &  3 & -3 &  0
\end{pmatrix},\ 
M^{-1} = \frac19 \begin{pmatrix} 
-4 & -1 &-1&  1 \\
-1 &  2 & 2&  1 \\
-1 &  2 & 2& -2 \\
 1 &  1 &-2&  8
\end{pmatrix}.
\]
One may calculate that
$\coker M \cong \Z/3 \oplus \Z/3 \oplus \Z/9$, so for $p = 3$ we have $k_1 = 0, {k_2 = k_3 = 1}$, ${k_4 = 2}$ and $q = 1$. Following the proof of the lemma, one may construct $T$ with $\det T = \pm 1$ such that
\[
\small
TMT^{\top} =
\begin{pmatrix} 
 -2 &   0 &  -3 &   0 \\
  0 &  -6 &   9 &  27 \\
 -3 &   9 & -12 & -27 \\
  0 &  27 & -27 & -90 
\end{pmatrix},\  
(TMT^{\top})^{-1} = 
\begin{pmatrix} 
  10  &   3 &   -7  &   3 \\
   3  & 4/3 &   -2  &   1 \\
  -7  &  -2 & 14/3  &  -2 \\
   3  &   1 &   -2  & 8/9
\end{pmatrix}
\]
satisfy the conditions on $\ord_3$ claimed in the lemma.
\end{example}
\begin{proof}[Proof of \cref{Lemma-matrixtransformation2}]
\cref{lem:rationalchange} applied to $N = M^{-1}$ with $\rho = 0$
yields a transformation matrix $S$.
Set $T = (S^{-1})^{\top}$, so that $(TMT^{\top})^{-1} = SM^{-1}S^{\top}$.
It follows that for the matrix $(TMT^{\top})^{-1}$,
the off-diagonal entries have $\ord_p \geq \rho = 0$.
Let $g_i$ denote ${-\ord_p}$ of the $i$-th diagonal entry.
We have to prove that $g_i = k_i$.

Let $\alpha = \ord_p (\det M)$.
On the one hand, we have $\alpha = k_1 +\dots +k_n$.
On the other hand, in the Leibniz formula for $\det (TMT^{\top})^{-1}$,
among the $n!$ summands the only one with minimal $\ord_p$ is
the product of the diagonal entries.
It follows that
$-\ord_p \det (TMT^{\top})^{-1} = \alpha = g_1 + \dots + g_n$.

Now let us use that $TMT^{\top}$ equals $(\det M)^{-1}$ times the adjugate of $(TMT^{\top})^{-1}$.
For $i\neq j$, the Leibniz determinant formula reveals that
the $(i,j)$-entry of the adjugate has $\ord_p \geq -\alpha + g_i + g_j$, and so the $(i,j)$-entry of $TMT^{\top}$ has $\ord_p \geq g_i + g_j$.
Similarly, the $i$-th diagonal entry of the adjugate has $\ord_p \geq -\alpha + g_i$, and so the $i$-th diagonal entry of $TMT^{\top}$ has $\ord_p \geq g_i$.
Let $R$ be the diagonal matrix with $i$-th diagonal entry $p^{g_i}$.
We have $\coker R \cong \Z/p^{g_1} \oplus \dots \oplus \Z/p^{g_n}$.
Note that $TMT^{\top}R^{-1}$ is an integer matrix with determinant $q$, i.e.~with determinant coprime with $p$.
It follows that the $p$-primary parts of the cokernels of the matrices $R$ and $(TMT^{\top}R^{-1})R$ are isomorphic. But the latter matrix equals $TMT^{\top}$, which has cokernel isomorphic to $\Z/p^{k_1} \oplus \dots \oplus \Z/p^{k_n}$.
Thus $g_i = k_i$ for all~$i$.
This concludes the proof.
\end{proof}

\begin{proposition}\label{prop:singdetwall}
Let $M \in \mathbb{Z}^{2g\times 2g}$ be a symmetric matrix with even diagonal entries and odd determinant.
Let $p$ be an odd prime.
Write $|\det(M)| = p^{\alpha} q$ with $\alpha\geq 0$ and $q$ not divisible by~$p$.
Let $m$ be the dimension of the kernel of $M$ over~$\mathbb{F}_p$.
Then
\[
\delta_p(M) = \legendre{q}{p} \cdot (-1)^{\sum_{k=1}^\infty r_{p,k}(\lambda_M)} \cdot (-1)^{\frac{p-1}{2}\bigl(\alpha+m+\frac{q-1}{2}\bigr)}.
\]
\end{proposition}
\begin{proof}
Let $0 \leq k_1 \leq k_2 \leq \dots \leq k_{2g}$ 
be the uniquely determined integers such that
\[
\coker M \cong \Z/p^{k_1} \oplus\dots\oplus \Z/p^{k_{2g}} \oplus \Z/q.
\]
Then $m$ equals the number of $k_i$ that are non-zero, i.e.\ 
$k_{2g - m} = 0$ and $k_{2g - m + 1} \geq 1$.
Let $s_{ij}, t_{ij}$ be coprime integers such that
the $(i,j)$-entry of $M^{-1}$ equals $s_{ij}/t_{ij}$.
By \cref{{Lemma-matrixtransformation2}}, we can w.l.o.g.\ assume that
$\ord_p (t_{ij})$
equals $0$ if $i\neq j$ 
and equals $k_i$ if~$i = j$.
Consider the submatrices
\begin{align*}
N & = M[1, \dots, 2g-m; 1, \dots, 2g-m], \\
N' & = M^{-1}[2g-m+1, \dots, 2g; 2g-m+1, \dots, 2g].
\end{align*}
By \cref{Lemma-Jacobi} (the General Jacobi Identity), we have
\begin{equation}\label{eq:jacobi}
\det(N) = \det(M)\cdot \det(N').
\end{equation}
Note that $\ord_p (\det(M)) = \alpha = - \ord_p(\det(N'))$.
It follows that $\det(N) \in \mathbb{Z}$ is not divisible by $p$.
Moreover, by \cref{{Lemma-matrixtransformation2}}, every entry of $M$ outside of $N$ is divisible by $p$.
So, $N$ corresponds to the submatrix in the definition of $\delta_p$
and thus
\begin{equation}\label{eq:deltap}
\delta_p(M) = \left(\frac{(-1)^{m + g}\cdot\det N}{p}\right).
\end{equation}
Let us now have a look at $\det(M)$.
Of course, $|\det(M)| = p^{\alpha} q$.
Moreover, the assumptions that $M$ has even diagonal entries and odd determinant implies that modulo 2, $M$ is a block sum of $g$ matrices of the form
\[
\begin{pmatrix} 0 & 1 \\ 1 & 0 \end{pmatrix}.
\]
Thus $\det(M) \equiv (-1)^g \pmod{4}$. From this it follows that
\begin{equation}\label{eq:detM}
\det(M) = (-1)^{\frac{\det M -1}{2} + g} p^{\alpha} q.
\end{equation}

Next, let us relate $\lambda_M$ to $\det(N')$.
For all $i\in\{2g-m+1,\ldots,2g\}$, we define
\[
u_i = p^{-k_i}\cdot\prod_{j=1}^{2g} t_{ij}.
\]
We have $\ord_p(u_i) = \ord_p(p^{-k_i}) + \sum_{j=1}^{2g} \ord_p(t_{ij}) = 0$.
Thus $u_i$ is an integer that is not divisible by $p$.
Let $f_i\in \Z^{2g}$ be defined as $u_i$ times the $i$-th standard basis vector.
Then
\[
\lambda_M(f_i, f_i) =
f_i^{\top} M^{-1} f_i = \frac{s_{ii}u_i^2}{t_{ii}}.
\]
We have $\ord_p(\lambda_M(f_i, f_i)) = \ord_p(s_{ii}) + 2\ord_p(u_i) - \ord_p(t_{ii})
 = 0 + 2\cdot 0 -k_i = -k_i$.
It follows that the order of $[f_i]$ in $\coker M$ equals $p^{k_i}$.
Moreover, for $i\neq j$,
\[
f_i^{\top} M^{-1} f_j = \frac{s_{ij} u_i u_j}{t_{ij}},
\]
which lies in $\mathbb{Z}$.
It then follows that $f_{2g-m+1}, \ldots, f_{2g}$ form an orthogonal basis of a
$\Z/p^{k_{2g-m+1}} \oplus \dots \oplus\Z/p^{k_{2g}}$ summand of $\coker M$.
On the summand $\Z/p^{k_i}$, the isometry type of $\lambda_M$ 
is $A_{p^{k_i}}$ or $B_{p^{k_i}}$ depending on the sign of
\[
\legendre{p^{k_i}\lambda_M(f_i, f_i)}{p} = 
\legendre{p^{k_i}s_{ii}u_i^2/t_{ii}}{p}.
\]
So we find, using the multiplicativity of the Legendre symbol
\[
(-1)^{\sum_{k = 1}^{\infty} r_{p,k}(\lambda_M)} =
\legendre{p^{\alpha}U^2\prod_{i=2g-m+1}^{2g} s_{ii} / t_{ii}}{p},
\]
where we write $U = \prod_{i=2g-m+1}^{2g} u_i$, which is an integer not divisible by $p$.
On the other hand, from the Leibniz formula for the determinant,
we see that the following integers are equivalent modulo $p$:
\[
p^{\alpha} U^2 \det(N') \equiv
p^{\alpha} U^2 \prod_{i={2g-m+1}} \frac{s_{ii}}{t_{ii}} \pmod{p}.
\]
Hence it follows that
\begin{equation}\label{eq:detc}
(-1)^{\sum_{k = 1}^{\infty} r_{p,k}(\lambda_M)} =
\legendre{p^{\alpha} U^2 \det(N') }{p}.
\end{equation}
Putting it all together, we find by using \eqref{eq:deltap}
\begin{align*}
\delta_p(M) & = \left(\frac{(-1)^{m + g}\cdot\det N}{p}\right)
 = \left(\frac{(-1)^{m + g}\cdot U^2\cdot \det N}{p}\right) \\
 \intertext{by the multiplicativity of the Legendre symbol and \eqref{eq:jacobi}}
 & = \legendre{(-1)^{m + g}}{p}
\legendre{p^{-\alpha}\det(M)}{p}
\legendre{p^{\alpha} U^2 \det(N')}{p} \\
\intertext{by \eqref{eq:detM} and \eqref{eq:detc}}
& =
\legendre{(-1)^{m + g}}{p} \cdot \legendre{(-1)^{\frac{\det M - 1}2 + g}q}{p} 
\cdot 
(-1)^{\sum_{k=1}^\infty r_{p,k}(\lambda_M)} \\
\intertext{again by multiplicativity of the Legendre symbol, and \eqref{eq:legendre-1}}
& =
\legendre{q}{p} \cdot 
(-1)^{\sum_{k=1}^\infty r_{p,k}(\lambda_M)}
\cdot
(-1)^{\frac{p-1}{2}(m + \frac{\det M - 1}2)}.
\end{align*}
This last expression equals the claimed expression for $\delta_p(M)$,
since
$m + \frac{\det M - 1}2 \equiv \alpha + m + \frac{q-1}2 \pmod{2}$
if $p \equiv 3\pmod{4}$.
\end{proof}

We now show that the obstruction given in \cref{thm:signedu} is equivalent to Lickorish's obstruction in the case that $L = K$ is a knot with ${u(K) = 1}$.
Lickorish's obstruction \cite{zbMATH03982071}, \cite[Proposition~2.1]{zbMATH04048673}, see also \cite[Theorem~4.5]{zbMATH06425399}%
\footnote{Each of the formulae given in \cite[Proposition~2.1]{zbMATH04048673} and \cite[Theorem~4.5]{zbMATH06425399} are equivalent to \eqref{eq:lickorish}, using different conventions.}
 says that if $K$ can be unknotted by changing a single crossing with sign~$\zeta$, then
there exists a generator $h$ of $H_1(D_K; \mathbb{Z})$ with%
\begin{equation}\label{eq:lickorish}
\lambda_{D_K}(h,h) = \frac{2\cdot\zeta\cdot (-1)^{\frac{\det K - 1}2}}{\det(K)} \in \mathbb{Q}/\mathbb{Z}.
\end{equation}
\begin{proposition}\label{rmk:lickorish}
For all knots $K$ and $\zeta \in \{\pm 1\}$,
the following are equivalent:
\begin{enumerate}[label=(\roman*)]
\item There is a generator $h$ of $H_1(D_K; \mathbb{Z})$ satisfying \cref{eq:lickorish}.
\item For all prime divisors $p$ of $\det K$ we have $d_p(K) = 1$, and
\begin{equation}\label{eq:lickorish2}
\begin{aligned}
& p \equiv 1 \,\Longrightarrow\, \delta_p(K) = 1,
&& p \equiv 3 \,\Longrightarrow\, \delta_p(K) = \zeta,\\
& p \equiv 5 \,\Longrightarrow\, \delta_p(K) = -1,
&& p \equiv 7 \,\Longrightarrow\, \delta_p(K) = -\zeta.
\end{aligned}
\end{equation}
\end{enumerate}
\end{proposition}
\begin{proof}
Rather than showing the implications (i) $\Rightarrow$ (ii) and
(ii) $\Rightarrow$ (i) separately, we are going to prove directly that (i) and (ii) are equivalent.
First, note that (i) and (ii) both imply that $H_1(D_K;\Z)$ is cyclic.
So we assume this throughout the proof. We also fix an arbitrary generator $h'$ of $H_1(D_K;\Z)$,
and write $\lambda_{D_K}(h', h') = a/\det K$, with $a$ and $\det K$ coprime.

Now, (i) is equivalent to the existence of an integer $b$ coprime with $\det K$
such that $h = bh'$ satisfies \eqref{eq:lickorish}, i.e.
\begin{align*}
\lambda_{D_K}(bh', bh') & = 2\zeta(-1)^{\frac{\det K - 1}2} / \det(K) \in \mathbb{Q}/\mathbb{Z} \Leftrightarrow \\
b^2 a & \equiv 2\zeta(-1)^{\frac{\det K - 1}2} \pmod{\det(K)},
\end{align*}
which is equivalent to
$c\coloneqq 2\cdot a \cdot \zeta\cdot  (-1)^{\frac{\det K - 1}2} \in \Z$ being a quadratic residue modulo $\det K$.
This is the case if and only if $c$ is a quadratic residue modulo $p$ for all prime divisors $p$ of $\det K$, which is equivalent to $\legendres{c}{p} = 1$.
In what follows, we show that for a fixed prime divisor $p$ of $\det K$, $\legendres{c}{p} = 1$ is equivalent to~\cref{eq:lickorish2}. This will imply that (i)~$\Leftrightarrow$~(ii) as claimed.

Write $\det K = p^{\alpha} q$ with $q$ not divisible by~$p$. Then, by
the multiplicativity of the Legendre symbol and \cref{eq:legendre-1},
\begin{equation}\label{eq:cp}
\legendre{c}{p} = \legendre{2 a \zeta\cdot  (-1)^{\frac{\det K - 1}2}}{p}
= \legendre{2 a \zeta}{p} \cdot  (-1)^{\frac{p - 1}2(\alpha + \frac{q-1}{2})}.
\end{equation}

Note that $q h'$ is a generator of the $\Z/p^{\alpha}$ summand of $H_1(D_K; \mathbb{Z})$.
Since
$\lambda_{D_K}(qh', qh') = qa / p^{\alpha}$,
the Wall isometry type of $\lambda_{D_K}$ on that summand
is $A_{p^{\alpha}}$ if $qa$ is a quadratic residue mod~$p^{\alpha}$, and $B_{p^{\alpha}}$ if it is a non-residue.
Since a number is a quadratic residue modulo~$p^{\alpha}$ if and only if it is a quadratic residue modulo~$p$,
it follows from \cref{prop:singdetwall} (with $m=1$) that
\[
\delta_p(K) = \legendre{q}{p} \cdot \legendre{qa}{p} \cdot (-1)^{\frac{p-1}{2}(\alpha+1+\frac{q-1}{2})}
= \legendre{a}{p} \cdot (-1)^{\frac{p-1}{2}(\alpha+1+\frac{q-1}{2})}.
\]
Combining this equation with \eqref{eq:cp}, one finds
that $\legendres{c}{p} = 1$ is equivalent to
\begin{align*}
\delta_p(K) & = \legendre{2\zeta}{p}\cdot(-1)^{\frac{p - 1}2(\alpha + \frac{q-1}{2})}\cdot (-1)^{\frac{p-1}{2}(\alpha+1+\frac{q-1}{2})} \Leftrightarrow \\ 
\delta_p(K) & = \legendre{2\zeta}{p}\cdot(-1)^{\frac{p - 1}2}.
\end{align*}
Considering the four possible remainders of $p$ modulo 8, one finds this to be equivalent to \cref{eq:lickorish2}.
\end{proof}

\section{Evaluations of knot polynomials}
\label{sec:eval}
\cref{Main-Theorem}, which we restate here, now follows from \cref{prop:singdetwall}.
\MainTheorem*
\begin{proof}
By Lipson's theorem, which states that $\varepsilon_K = \delta_3(K)$, where $\varepsilon_K$ is as in \eqref{eq:atomega} and 
$\delta_3$ is the singular determinant from \cref{def:singdetL},
we have
\[
V_K(e^{2\pi i/6}) = \delta_3(K) \cdot (i\sqrt{3})^{\dim H_1(D_K; \mathbb{F}_3)}.
\]
Now, choose a symmetrized Seifert matrix $M$ for $K$.
By \cref{prop:singdetwall}, we then have
\[
\delta_3(K) = \delta_3(M) = \legendre{q}{3} 
 \cdot (-1)^{\sum_{k=1}^\infty r_{3,k}(M)} \cdot (-1)^{\alpha+\dim H_1(D_K; \mathbb{F}_3)+\frac{q-1}{2}}.
\]
To conclude the proof, one
combines these two formulae and observes that
\[
\legendre{q}{3}\cdot (-1)^{\frac{q-1}{2}} = \nu(q),
\]
for $\nu$ as defined above.
\end{proof}

An analogous result holds for the $Q$-polynomial. Like the Jones polynomial, the $Q$-polynomial is uniquely characterized by a skein relation, but with the difference that its relation applies to unoriented links. 
Namely, the \emph{$Q$-polynomial} $Q_L(z) \in \Z[z^{\pm 1}]$ for an unoriented link is defined by the value~$1$ for the unknot and the skein relation~\cite{zbMATH03958240,MR2635346}
\[
Q_{L_+}(z) + Q_{L_-}(z)
=
z(Q_{L_0}(z) + Q_{L_\infty}(z)),
\]
where $L_+, L_-, L_0, L_\infty$ are unoriented links that locally differ as follows:
\begin{figure}[h]
\centering
\includegraphics[scale=.3]{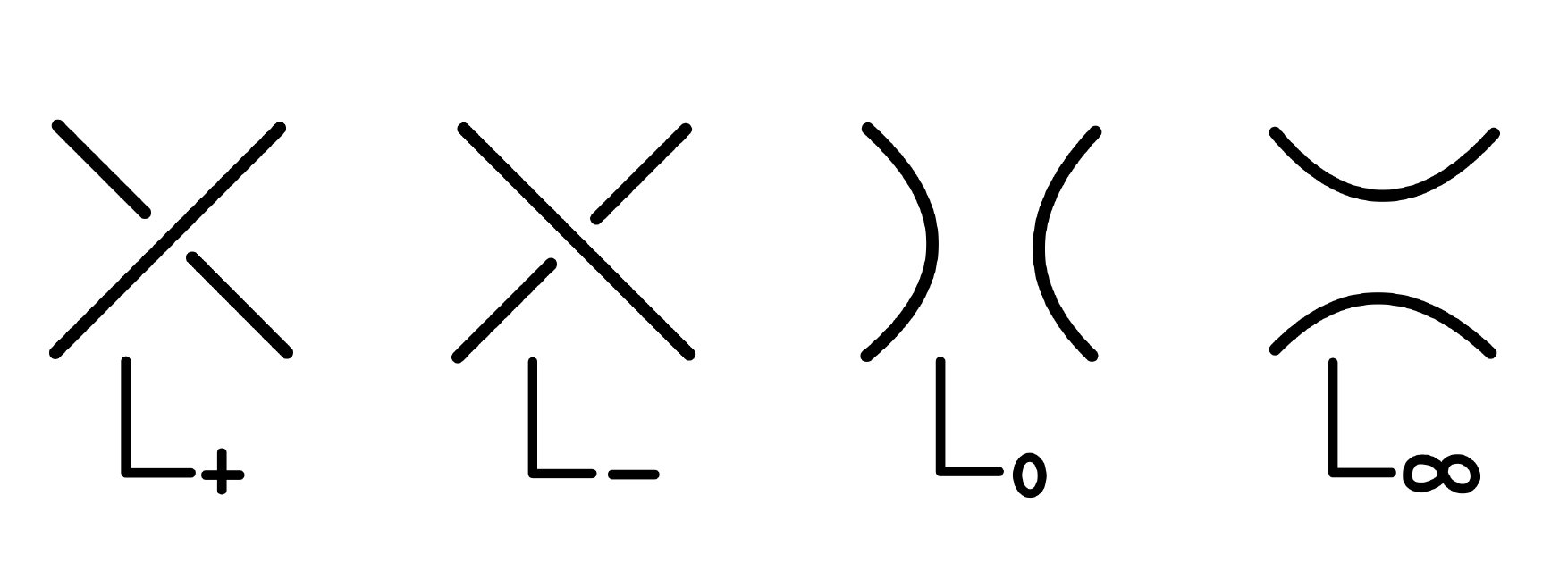}
\end{figure}

In the literature, the $Q$-polynomial is sometimes defined as the Kauffman polynomial evaluated at~$1$ in its first variable~\cite{zbMATH04184437,zbMATH04138782}. 
It is well-known that
\[
Q_L\Bigl(\frac{\sqrt{5}-1}{2}\Bigr) = \pm (\sqrt{5})^{\dim H_1(D_L; \mathbb{F}_5)}.
\]
In fact, Rong proved for all links~$L$ that~\cite{zbMATH04184437}
\begin{equation}\label{eq:rong}
Q_L\Bigl(\frac{\sqrt{5}-1}{2}\Bigr)
=
\delta_5(L)
\cdot
(\sqrt{5})^{\dim H_1(D_L; \mathbb{F}_5)}.
\end{equation}

The following theorem expresses the sign on the right-hand side in terms of the double branched covering of~$L$, in the case that $L$ is a knot.
\begin{theorem}\label{thm:q}
Let a knot $K$ be given. Denote by $q$ the unique non-negative integer such that $q$ is not divisible by~$5$ and $\det(K) = 5^\alpha q$, for some $\alpha \geq 0$.
Then we have
\[
Q_K\Bigl(\frac{\sqrt{5}-1}{2}\Bigr)
=
\legendre{q}{5}\cdot 
(-1)^{\sum_{k=1}^\infty \, r_{5,k}(K)}\cdot
(\sqrt{5})^{\dim H_1(D_K; \mathbb{F}_5)}.
\]
\end{theorem}
\begin{proof}
We proceed similarly as in the proof of \cref{Main-Theorem}:
namely, the statement follows by combining
Rong's formula~\cref{eq:rong}
with
\cref{prop:singdetwall}, which for $p = 5$ reads
\[
\delta_5(K) = 
\legendre{q}{5}\cdot 
(-1)^{\sum_{k=1}^\infty \, r_{5,k}(K)}.\qedhere
\]
\end{proof}
\cref{rmk:lickorish} and Rong's formula~\cref{eq:rong} also lead us to
a counterexample to the following conjecture of Stoimenow.
\begin{conjecture}[{\cite[Conjecture~7.3]{zbMATH02152374}}]\label{conj:stoimenow}
Let $K$ be a knot with cyclic $H_1(D_K; \mathbb{Z})$ and $\det K$ divisible by~$5$.
Then
\[
Q_K\Bigl(\frac{\sqrt{5}-1}{2}\Bigr)
=
\left\{
\begin{array}{ll}
-\sqrt{5} & \text{ if } \exists h \in H_1(D_K; \mathbb{Z}): \lambda_{D_K}(h,h) = \pm 2 / \det K, \\
+\sqrt{5} & \text{ if } \nexists h \in H_1(D_K; \mathbb{Z}): \lambda_{D_K}(h,h) = \pm 2 / \det K. \\
\end{array}
\right.
\]
\end{conjecture}
\begin{example}[\bfseries Counterexample for \cref{conj:stoimenow}]
\label{ex:stoimenow}
Let $K$ be a knot with Seifert matrix 
$
\bigl(\begin{smallmatrix}
22	&	17\\ 17	&	22
\end{smallmatrix}\bigr)
$,
such as the pretzel knot $P(5,17,5)$.
Then, one computes $\det K = 195 = 3\cdot 5\cdot 13$, $\delta_5(K) = -1$ and $\delta_{13}(K) = 1$.
On one hand, \cref{eq:rong} now implies that $ Q_K(\tfrac{\sqrt{5}-1}{2}) = - \sqrt{5}$.
On the other hand,
any $h\in H_1(D_K; \mathbb{Z})$ satisfying
\[
\lambda_{D_K}(h,h) = \frac{\pm 2}{\det(K)}
\]
has order $\det(K)$, i.e.~is a generator;
but by \cref{rmk:lickorish}, there does not exist such a generator.
So $K$ is a counterexample to the above conjecture.
Further counterexamples with the same values for the determinant, $\delta_5$ and $\delta_{13}$ (but genus greater than 1 and thus larger Seifert matrices)
can also be found in the knot tables \cite{knotinfo}, namely $12a628$,
$12a665$,
$12a828$, and
$12a1044$.
\end{example}

\section{The case of links}\label{sec:links}
In this section, we discuss to what extent the results that we stated for knots generalize to multi-component links~$L$.
We will need the following link invariants; details about all of which can be found in introductory textbooks such as~\cite{zbMATH01092415,zbMATH00907055}. 
Denote by $c(L)$ the number of components of~$L$;
let $\sigma(L)$ be the classical knot signature of~$L$, i.e.~the number of positive eigenvalues minus the number of negative eigenvalues of a symmetrized Seifert matrix $M$ of~$L$; 
let $\det(L) = |\det(M)|$ be the determinant of~$L$, which equals the order of $H_1(D_L;\mathbb{Z})$ if that order is finite, and $0$ otherwise;
and let $\Arf(L)\in\{0,1\}$ be the Arf invariant of $L$, which is
only defined if $L$ is \emph{proper}, i.e.~each component $K$ of $L$ has even linking number with $L\setminus K$.
The following is known about the values of the Jones polynomial at roots of unities~\cite{zbMATH03899758,zbMATH03977937,zbMATH03982072}:
\begin{align}\label{eq:eval1link}
V_L(1)            & = (-2)^{c(L) - 1}, \\ 
V_L(-1)           & = i^{\sigma(L)}\det(L), \label{eq:eval-1link} \\
V_L(e^{2\pi i/3}) & = (-1)^{c(L) - 1}, \\
V_L(i)            & = \left\{\begin{array}{rl}
(-\sqrt{2})^{c(L) - 1}(-1)^{\operatorname{Arf}(L)}  & \text{if $L$ is proper,} \\[1ex]
0 &\text{else},
\end{array}\right. \\
V_L(e^{2\pi i/6}) & = \varepsilon_L i^{c(L)-1} (i\sqrt{3})^{\dim H_1(D_L; \mathbb{F}_3)}. \label{eq:eval6link}
\end{align}
We have not found the explicit equation \cref{eq:eval-1link} for links in the literature, so for completeness, we provide a proof below in \cref{prop:alex}.

\begin{remark} We would like to give two short comments on conventions in order to explain why \eqref{eq:eval1link}--\eqref{eq:eval6link} are stated slightly differently in some of the literature.
Firstly, to evaluate a polynomial in $t^{1/2}$ at $t = e^{si}$ for some $s\in[0,2\pi)$ is understood as evaluating $t^{1/2} = e^{si/2}$.
The alternative evaluation $t^{1/2} = -e^{si/2}$ yields the extra factor $(-1)^{c(L)-1}$.
Secondly, Jones's original skein relation differs from ours in \eqref{eq:skeinjones}
by switching $L_+$ and $L_-$. Following Jones's convention, one obtains for the link $L$ the polynomial $(-1)^{c(L)-1} V_L(t^{-1})$. To be precise, $V_L(t^{-1})$ is understood as the polynomial obtained from $V_L$ by replacing 
$t^{1/2}$ with~$t^{-1/2}$.
\end{remark}

Lipson proved that $\varepsilon_L = \delta_3(L)$ holds for all links $L$~\cite{Lipson}.
Combining \cref{Main-Theorem} with \cref{thm:signedu} yields the following corollary, which is due to Traczyk~\cite{zbMATH01341751} in the case of knots.%
\footnote{Note that in \cite[Theorem~3.1]{zbMATH01341751}, $u_{+-}$ should be replaced by $u_{-+}$: changing negative crossings switches the sign, while changing positive crossings does not. The reason for the mistake appears to be that the terms $\pm i(\sqrt{3})^k$ in the proof should be replaced $\pm (i\sqrt{3})^k$. Fortunately, applications of Traczyk's theorem as in \cite{zbMATH05242998} remain valid.}
\begin{corollary}\label{cor:traczyk}
Let $L$ be a link with $c(L)$ components and unknotting number
\[
u(L) = \dim H_1(D_L; \mathbb{F}_3) - c(L) + 1.
\]
Let $u(L) = u_+ + u_-$ such that $L$ can be unknotted by changing $u_+$ positive crossings and $u_-$ negative crossings. Then
\[
V_L(e^{2\pi i /6}) = (-1)^{u_-} \cdot i^{c(L) - 1}\cdot (i\sqrt{3})^{\dim H_1(D_L; \mathbb{F}_3)}.
\]
\end{corollary}

While $V_L(1)$ and $V_L(e^{2\pi i/3})$ are determined by the double branched covering, the following example shows that $V_L(-1)$ and $V_L(e^{2\pi i/6})$ are not.

\begin{figure}[bt]
\centering
    \includegraphics[scale=.5]{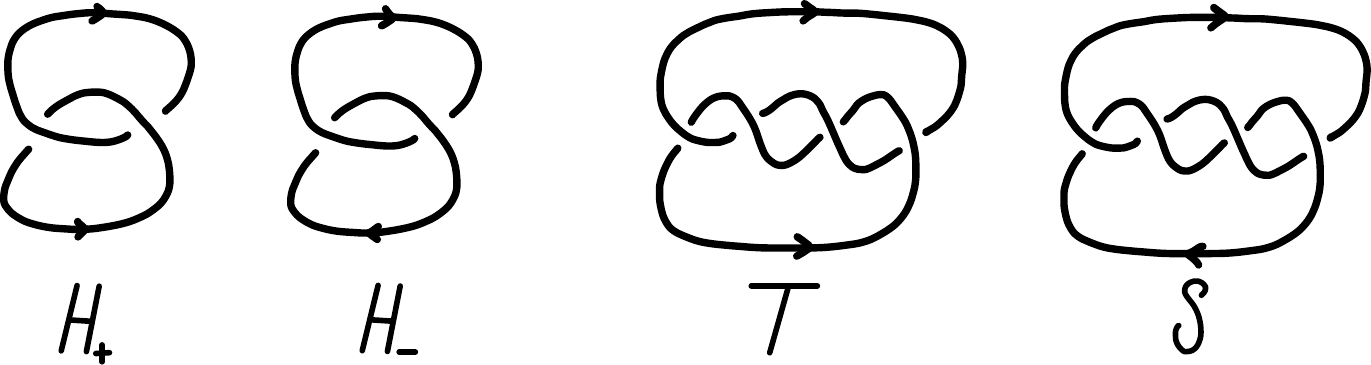}
\caption{The links appearing in \cref{ex:counterex}.}
\label{fig:counterx}
\end{figure}

\begin{example}[\bfseries\cref{Main-Theorem} cannot be extended to links]
\label{ex:counterex}
Let $H_+ = T(2,2)$ and $H_- = T(2,-2)$ denote the positive and the negative Hopf link, respectively (see~\cref{fig:counterx}).
These two links have homeomorphic double branched coverings, namely $D_{H_+} \cong D_{H_-} \cong \mathbb{R}P^3$.
However, their Jones polynomials $V_{H_+}(t) = -t^{5/2}-t^{1/2}$ and $V_{H_-}(t) = -t^{-5/2}-t^{-1/2}$ evaluate as
\begin{align*}
V_{H_+}(-1) = -2i & \neq 2i = V_{H_-}(-1), \\
V_{H_+}(e^{2\pi i/6}) = -i & \neq i = V_{H_-}(e^{2\pi i/6}).
\end{align*}

One checks that these evaluations are in accordance with the formulae given above,
since $\det H_{\pm} = 2$ and ${\dim H_1(D_{H_{\pm}}; \mathbb{F}_3) = 0}$, as can be seen from
$H_1(D_{H_{\pm}}; \mathbb{Z}/2) \cong \mathbb{Z}/2$.
Moreover, $M = (\mp 2)$ is a symmetrized Seifert matrix for~$H_{\pm}$, and thus
$\sigma(H_{\pm}) = \mp 1$ and $\delta_3(H_{\pm}) = \mp 1$.

Similarly, denote by $T$ the positive torus link $T(2,4)$, and by $S$ the link obtained from $T$ by switching the orientation of one component (see~\cref{fig:counterx}). 
Since the double branched covering does not depend on the orientations of link components, these two links have homeomorphic double branched coverings (homeomorphic to the lens space~$L(4,3)$).
However, their Jones polynomials $V_T(t) = -t^{3/2} - t^{7/2} + t^{9/2} - t^{11/2}$ and $V_S(t) = -t^{-9/2}-t^{-5/2}+t^{-3/2}-t^{-1/2}$ evaluate as
\[
V_T(i)            = \sqrt{2} \neq -\sqrt{2} = V_S(i),
\]
which is, again, in accordance with the formulae given above since $\operatorname{Arf}(T) = 1$ and $\operatorname{Arf}(S) = 0$.
\end{example}

For the $Q$-polynomial (see \cref{sec:eval}),
combining \cref{thm:q} with \cref{thm:signedu}
yields the following result,
which was proven by Stoimenow for knots~\cite[Theorem~4.1]{zbMATH02152374}.
\begin{corollary}
If the $Q$-polynomial of a link $L$ with $c(L)$ components satisfies 
\[
Q_L\Bigl(\frac{\sqrt{5}-1}{2}\Bigr) =
(-1)^{a+c(L)} (\sqrt{5})^a,
\]
for some non-negative integer~$a$, then $u(L) > a - c(L) + 1$.
\qed
\end{corollary}

Finally, let us provide a proof of \cref{eq:eval-1link}.
We will use the \emph{Alexander polynomial} $\Delta_L(t) \in \mathbb{Z}[t^{\pm 1}]$, which may be defined via its value of $1$ for the unknot, and the skein relation
\begin{equation}\label{eq:skeinalex}
\Delta_{L_+}(t) - \Delta_{L_-}(t) = (t^{1/2} - t^{-1/2}) \Delta_{L_0}(t).
\end{equation}
Equivalently, $\Delta_L$ may be defined as
\[
\Delta_L(t) = 
\det(-t^{1/2} A + t^{-1/2} A^{\top}),
\]
where $A$ is any (unsymmetrized) Seifert matrix of~$L$ (we omit the definition of such a matrix~$A$, and merely note that $M = A + A^{\top}$ is a symmetrized Seifert matrix as in \cref{sec:singdet}). 
The above definitions of $\Delta_L$ give the Conway normalization of the Alexander polynomial;
in particular, $\Delta_L$ is well-defined on the nose, and not just up to multiplication with a unit of the ring~$\mathbb{Z}[t^{\pm 1}]$.
Now, we have the following result (for knots, compare with \cite[Lemma~4.9]{zbMATH06425399}).
\begin{proposition}\label{prop:alex}
For all links $L$,
\begin{align}
V_L(-1)      & = i^{\sigma(L)}\det(L), \tag{\ref{eq:eval-1link}}
 \\
\Delta_L(-1) & = i^{-\sigma(L)}\det(L), \label{eq:alex}\\
V_L(-1)      & = (-1)^{c(L) - 1} \Delta_L(-1). \label{eq:alexjones}
\end{align}
\end{proposition}

\begin{proof}
Evaluating the skein relations
of the Jones polynomial \eqref{eq:skeinjones} and
the Alexander polynomial \eqref{eq:skeinalex}
at $t = -1$ (i.e.~at~$t^{1/2} = i$) 
yields the same equation,
except that $L_+$ and $L_-$ are switched. From this follows~\cref{eq:alexjones}.

Note that \cref{eq:eval-1link} follows from \cref{eq:alex} and~\cref{eq:alexjones}. So it just remains to establish~\cref{eq:alex}.
Let $A$ be an $n\times n$ Seifert matrix of~$L$. 
We have
\[
\Delta_L(-1) = \det(-iA -iA^{\top}) = i^{-n} \det(M)
\]
for $M = A + A^{\top}$.
If $\det(M) = 0$, then $\Delta_L(-1) = 0$ and
$\det(L) = {|\det(M)| = 0}$, so the equality holds.
The case ${\det(M) \neq 0}$ remains.
Denote by $e_{\pm}$ the number of positive and negative eigenvalues of $M$, respectively. 
The sign of $\det(M)$ is $(-1)^{e_-}$.
Since $e_+ + e_- = n$ and $e_+ - e_- = \sigma(M) = \sigma(L)$,
we have $e_- = (n - \sigma(L))/2$.
Thus
\[
\det(M) = (-1)^{(n - \sigma(L))/2} \det(L).
\]
It follows that
\[
\Delta_L(-1) = i^{-n} \det(M) = 
i^{-\sigma(L)} \det(L).
\]
This concludes the proof.
\end{proof}

\bibliographystyle{myamsalpha}
\bibliography{References}
\end{document}